\documentclass{amsart}

\usepackage[foot]{amsaddr}

\usepackage{xcolor}

\usepackage[lite,initials,nobysame]{amsrefs}

\usepackage{amsfonts,amsmath,amssymb,enumerate,mathrsfs}
\usepackage{hyperref}
\usepackage{url}

\numberwithin{equation}{section}

\newtheorem{thm}{Theorem}[section]
\newtheorem{cor}[thm]{Corollary}
\newtheorem{lem}[thm]{Lemma}
\newtheorem{prop}[thm]{Proposition}

\theoremstyle{definition}
\newtheorem{defn}[thm]{Definition}
\newtheorem{rem}[thm]{Remark}
\newtheorem*{notn}{Notation}

\newcommand{\olR}{\overline{R}}
\newcommand{\olM}{\overline{M}}
\renewcommand\le{\leqslant}
\renewcommand\ge{\geqslant}

\newcommand{\msJ}{\mathscr{J}}
\newcommand{\F}{\mathbb{F}}
\newcommand{\Z}{\mathbb{Z}}
\newenvironment{augmat}[1]{%
  \left(\begin{array}{@{}*{#1}{c}|c@{}}
}{%
  \end{array}\right)
}

\usepackage{tikz}
\usetikzlibrary{decorations.markings}
\usetikzlibrary{arrows}

\tikzstyle{vertex}=[circle,draw,inner sep=0pt, minimum size=6pt]

\usepackage{colonequals}

\begin{document}

\title{Covering rings by proper ideals}

\author[M. H. W. Chen]{Malcolm Hoong Wai Chen} \address{Department of Mathematics, University of Manchester, Manchester M13 9PL, UK} \email{malcolmhoongwai.chen@manchester.ac.uk}

\author[E. Swartz]{Eric Swartz} \address{Department of Mathematics, William \& Mary, Williamsburg, VA 23187, USA} \email{easwartz@wm.edu}

\author[N. J. Werner]{Nicholas J. Werner} \address{Department of Mathematics, Computer and Information Science, SUNY at Old Westbury, Old Westbury, NY 11568, USA} \email{wernern@oldwestbury.edu}


\begin{abstract}
A \emph{cover by left ideals} of an associative (not necessarily commutative or unital) ring $R$ is a collection of proper left ideals whose set-theoretic union equals $R$. If such a cover exists, then $\eta_\ell(R)$ is the cardinality of a minimal cover, and $R$ is $\eta_\ell$-elementary if $\eta_\ell(R)<\eta_\ell(R/I)$ for every nonzero two-sided ideal $I$ of $R$. We classify all $\eta_\ell$-elementary rings, and determine their covering numbers. Covers by right or two-sided ideals are also studied. This completely characterizes rings admitting finite covers by ideals. Our results generalize to finite covers of modules by submodules, and we determine all possible covering numbers.
\end{abstract}

\maketitle

\section{Introduction}

In this paper, we assume rings to be associative, but not necessarily commutative nor to have a multiplicative identity. Our goal is to determine when such a ring can be expressed as a finite union of its proper left, right, or two-sided ideals, as well as the cardinality of such a minimal cover. 

The motivation for this problem comes from group theory. A group $G$ is said to be \emph{coverable} if there exists a collection of proper subgroups of $G$ whose union is equal to $G$. It is an easy exercise to show that a group is never the union of two of its proper subgroups, and a classical result due to Scorza \cite{Scorza_1926} shows that a group is the union of three proper subgroups if and only if it has a quotient isomorphic to the Klein four-group. Later, Cohn \cite{Cohn_1994} began to systematically study the \emph{covering number} $\sigma(G)$ of a group $G$, which is the minimum number of subgroups required to cover $G$. Cohn showed that there exists a group with covering number $p^d+1$ for every prime $p$ and integer $d \ge 1$, and Tomkinson \cite{Tomkinson_1997} showed that there is no group with covering number seven. The covering numbers of groups have been studied extensively, and all integers $N \le 129$ that are not the covering number of a group were determined in \cite{Garonzi_Kappe_Swartz_2022}, which also contains a survey of the related work for groups.

This motivated inquiry into a broader question: What are the covering numbers of other algebraic structures? This has been studied for vector spaces \cite{Clark_2012,Khare_2009}, modules (over commutative rings) \cite{Ghosh_2022,Khare_Tikaradze_2022}, loops \cite{Gagola_Kappe_2016}, and semigroups \cite{Donoven_Kappe_2023}. We refer the reader to the survey article \cite{Kappe_2014} for more details. 

The related question for rings has also received significant attention. This was first initiated by Lucchini and Mar\'{o}ti \cite{Lucchini_Maroti_2012}, who determined all rings with covering number three. Building upon a series of earlier work \cite{Cai_Werner_2019,Crestani_2012,Peruginelli_Werner_2018,Swartz_Werner_2021,Swartz_Werner_2024a}, the recent article \cite{Swartz_Werner_2024b} completely determined which integers can occur as covering numbers of rings. For a ring $R$, let $\sigma(R)$ be the minimum number of proper subrings of $R$ necessary to cover $R$ (if such a cover exists). The approach in \cite{Swartz_Werner_2024b} is based on the observation that if $R$ has a two-sided ideal $I$ such that $R/I$ has a cover, then it can be lifted to a cover of $R$, so $\sigma(R) \le \sigma(R/I)$. A ring $R$ is called \emph{$\sigma$-elementary} if the inequality $\sigma(R) < \sigma(R/I)$ is strict for all $I \neq \{0\}$. Since any coverable ring is either $\sigma$-elementary, or has a proper $\sigma$-elementary residue ring with the same covering number, knowledge of $\sigma$-elementary rings allows us to characterize coverable rings and determine their covering numbers. A classification of all $\sigma$-elementary rings is given in \cite[Thm.\ 1.3]{Swartz_Werner_2024b}.

While many questions on the covering numbers of groups remain open, if one restricts to covers by \emph{normal} subgroups there is an elegant result \cite[Thm.\  1]{Brodie_Chamberlain_Kappe_1988} (see also \cite{Bhargava_2002}): a group $G$ is the union of proper normal subgroups if and only if it has a quotient isomorphic to $C_p \times C_p$ (the direct product of two cyclic groups of prime order $p$); the corresponding covering number $\eta(G)$ is $p+1$ for the smallest such $p$. Our investigations in this paper are prompted by the ring-theoretic analogue, namely covers by proper ideals. 

Motivated by the preceding discussions, we make the following definitions.

\begin{defn}
Let $R$ be a ring. The \emph{left} (resp.\ \emph{right, two-sided}) \emph{ideal covering number} $\eta_\ell(R)$ (resp.\ $\eta_r(R)$, $\eta(R)$) is the cardinality of a minimal cover of $R$ by proper left (resp.\ right, two-sided) ideals, if such a cover exists, and $\infty$ otherwise. 
 
We say that $R$ is \emph{$\eta_\ell$-elementary} (resp.\ $\eta_r$- or \emph{$\eta$-elementary}) if $\eta_\ell(R) < \eta_\ell(R/I)$ (resp.\ $\eta_r(R)<\eta_r(R/I)$, $\eta(R)<\eta(R/I)$) for every nonzero two-sided ideal $I$ of $R$.

We use the notation $\eta_*(R)$ and \emph{$\eta_*$-elementary} in a statement which would apply to any of $\eta_\ell$, $\eta_r$, or $\eta$. Note that an $\eta_*$-elementary ring is coverable by the relevant type of ideal, which we will refer to as \emph{$*$-ideals}.
\end{defn}

We remark that any unital ring cannot be covered by proper ideals, as any ideal containing the multiplicative identity is the whole ring. It follows that $\eta_*$-elementary rings are necessarily nonunital. However, as we shall see, any $\eta_\ell$-elementary (resp.\ $\eta_r$-elementary) ring that is not $\eta$-elementary contains a left (resp.\ right) multiplicative identity.

\begin{notn}
Let $C_p$ denote the cyclic group of order $p$, $\F_q$ the finite field of order $q$ (a prime power), and $M_n(\F_q)$ the ring of $n \times n$ matrices with entries from $\F_q$. 
\end{notn}

Covers by two-sided ideals were first considered independently by O'Neill \cite{ONeill_1980} and Parmenter \cite{Parmenter_1994}; in each paper, the authors provide characterizations of rings with covers by two-sided ideals. In contrast, covers by left (resp.\ right) ideals were, to the best of our knowledge, first studied in the preprint \cite{Chen_2025}, which suggested that $$\left\{\begin{pmatrix} a & b \\ 0 & 0 \end{pmatrix} : a,b \in \F_p \right\} \le M_2(\F_p)$$ is an $\eta_\ell$-elementary ring for every prime $p$. In this present work, we generalize these results, providing a complete classification of $\eta_*$-elementary rings in each case as well as determining their ideal covering numbers. A ring $R$ that is coverable by a finite number of $*$-ideals has an $\eta_*$-elementary residue ring with the same ideal covering number. Hence, our classification completely characterizes rings that are coverable by a finite number of proper ideals, and determines all possible ideal covering numbers of such rings.

Since every commutative $\eta_\ell$- and $\eta_r$-elementary ring is also $\eta$-elementary, it suffices to consider the noncommutative case when studying covers by one-sided ideals. Our main theorems are the following.

\begin{thm}\label{thm:two-sided}
    $R$ is an $\eta$-elementary ring if and only if $R^2=\{0\}$ and $(R,+) \cong C_p \times C_p$ for some prime $p$. Moreover, $\eta(R)=p+1$.
\end{thm}

From Theorem \ref{thm:two-sided}, we recover \cite[Thm.\ 1]{ONeill_1980} and \cite[Thm.\  2]{Parmenter_1994}: $R$ is coverable by finitely many two-sided ideals if and only if $R$ has a finite homomorphic image $S$ such that $S^2=\{0\}$ and $(S,+)$ is a nontrivial noncyclic group.  We remark that Theorem \ref{thm:two-sided} can be proved independently by appealing to \cite[Thm.\ 1]{ONeill_1980} or \cite[Thm.\ 2]{Parmenter_1994}---which reduces the problem to covering a noncyclic abelian group---and applying a result of Cohn \cite[Thm.\ 2]{Cohn_1994}.  We include our proof of Theorem \ref{thm:two-sided} for completeness.

The situation for rings with a cover by left or right ideals is far more interesting.

\begin{thm} \label{thm:main}
Let $R$ be a noncommutative ring.
    \begin{enumerate} \itemsep0.5em
    \item $R$ is $\eta_\ell$-elementary if and only if $$R \cong \left\{ \begin{pmatrix} A & v \\ 0 & 0 \end{pmatrix} : A \in M_n(\F_q) \text{ and } v \in \F_q^n \right\} \le M_{n+1}(\F_q)$$ for some prime power $q$ and integer $n \ge 1$. We have $\eta_\ell(R)=\frac{q^{n+1}-1}{q-1}$.
    \item $R$ is $\eta_r$-elementary if and only if $$R \cong \left\{ \begin{pmatrix} A & 0 \\ v^T & 0 \end{pmatrix} : A \in M_n(\F_q) \text{ and } v \in \F_q^n \right\} \le M_{n+1}(\F_q)$$ for some prime power $q$ and integer $n \ge 1$. We have $\eta_r(R)=\frac{q^{n+1}-1}{q-1}$.
    \end{enumerate}
\end{thm}

Observe that it suffices to prove (1), since the proof of (2) follows by passing to the opposite ring. Moreover, if $R$ is the ring of Theorem \ref{thm:main}(1) and $S = M_n(\F_q)$, then $R$ is a left $S$-module. Under this identification, the left $S$-submodules of $R$ are exactly the left ideals of $R$ (see Lemma \ref{lem: ideals and modules}). Thus, we have produced the first example of a module over a noncommutative ring with covering number $\frac{q^{n+1}-1}{q-1}$.

\begin{cor}\label{cor: module coverings}
For every prime power $q$ and integer $n \ge 1$, there exists a left (resp.\ right) $M_n(\F_q)$-module $M$ such that the covering number of $M$ by left (resp.\ right) $M_n(\F_q)$-submodules is $\frac{q^{n+1}-1}{q-1}$.
\end{cor}

This generalizes the known results on covers of modules over commutative rings. In \cite[Thm.\ 2.2]{Khare_Tikaradze_2022}, it is shown that if $R$ is a commutative unital ring and $M$ is an $R$-module that admits a finite cover by proper $R$-submodules, then the covering number of $M$ is $q+1$ for some prime power $q$. We are able to generalize these results to modules over noncommutative rings, and prove that the covering numbers obtained in Corollary \ref{cor: module coverings} are the only possible finite covering numbers for modules.

\begin{thm}\label{thm: all module covers}
Let $R$ be a unital ring and let $M$ be a left (resp.\ right) $R$-module that admits a finite cover by proper left (resp.\ right) $R$-submodules. Then, there exists some prime power $q$ and integer $n \ge 1$ such that the covering number of $M$ by left (resp.\ right) $R$-modules is $\tfrac{q^{n+1}-1}{q-1}$.
\end{thm}

\subsection*{Outline}
In Section \ref{sec: prelim}, we review how to embed a nonunital ring into a unital ring, and use this embedding to describe the structure of finite nonunital rings in characteristic $p$. In Section \ref{sec: reduction}, we develop basic properties of $\eta_\ell$- and $\eta$-elementary rings. We prove Theorem \ref{thm:two-sided}, and build towards Proposition \ref{prop:structR}, which shows that a noncommutative $\eta_\ell$-elementary ring $R$ admits a decomposition $R=S \oplus J$, where $S$ is a finite unital semisimple ring of characteristic $p$ and $J$ is the Jacobson radical of $R$. The multiplication on $R$ is given by $(s_1+j_1)(s_2+j_2) = s_1s_2 + s_1j_2$, so that
\begin{equation}\label{eq: R matrix form}
R \cong \left\{\begin{pmatrix}
                    s & j \\
                    0 & 0
                   \end{pmatrix} : s \in S \text{ and } j \in J \right\}.
\end{equation}
To prove Theorem \ref{thm:main}(1), we show that $S \cong M_n(\F_q)$ and $J \cong \F_q^n$ for some prime power $q$ and integer $n \ge 1$. In Section \ref{sec: new numbers}, we show that when $S$ and $J$ have these forms, the ring $R$ in \eqref{eq: R matrix form} is $\eta_\ell$-elementary and $\eta_\ell(R)=\frac{q^{n+1}-1}{q-1}$. In Section \ref{sec: structure}, we prove that these are the only noncommutative $\eta_\ell$-elementary rings. In Section \ref{sec: modules}, we close the paper with the proof of Theorem \ref{thm: all module covers}.

\section{Preliminaries} \label{sec: prelim}

It is well known that a nonunital ring $R$ of characteristic $p$ can be embedded into a unital ring $R'$: let $R' := \F_p \times R$ with multiplication given by the rule
 \[ (n_1, r_1)(n_2, r_2) := (n_1n_2, n_1r_2 + n_2r_1 + r_1r_2).\]
Notice that, for any $(n, r) \in R'$, $(1,0)(n,r) = (n,r)(1,0) = (n,r)$, so $R'$ has a multiplicative identity, and there is a natural identification of $R$ with $\{0\} \times R \subseteq R'$, so we may view $R$ as a subring of $R'$.  Such an extension $R'$ is sometimes called the \emph{Dorroh extension} of $R$, although this name is sometimes reserved for the extension $\Z \times R$ with analogous multiplication. Throughout this paper, $R'$ will always refer to the extension of $R$ defined above.

We prove one basic result that will be used later.

\begin{lem}
 \label{lem:idealsofDorroh}
If $I$ is any two-sided ideal of $R$, then $\{0\} \times I$ is a two-sided ideal of $R'$.  In particular, $\{0\} \times R$ is a two-sided ideal of $R'$.
\end{lem}

\begin{proof}
 For any $(n, r) \in R'$ and $(0, i) \in \{0\} \times I$, we have $(n,r) \cdot (0,i) = (0, ni + ri)$ and $(0,i) \cdot (n,r) = (0, ni + ir)$, both of which are in $\{0\} \times I$. The result follows.
\end{proof}

If $R$ is finite, then the structure of $R'$ is determined by the Wedderburn-Malcev Theorem (sometimes also called the Wedderburn Principal Theorem); see e.g. \cite[Sec.\ 11.6, Cor.\ p.\ 211]{Pierce_1982} or \cite[Thm.\ VIII.28]{McDonald_1974}. To describe the theorem, we need to define the Jacobson radical of a ring. There are a number of common equivalent definitions for the Jacobson radical of a unital ring. For a nonunital ring, the definition is more obscure. We follow the approach taken in \cite[Sec.\ 4]{Lam_1991}.

If we define the operation $\circ$ on a (not necessarily unital) ring $T$ by $a \circ b \colonequals a + b - ab$, then $(T, \circ)$ is a monoid with $0$ as the identity element. An element of $T$ is called \emph{left (\textnormal{resp.}\ right) quasi-regular} if it has a left (resp.\ right) inverse in the monoid $(T, \circ)$. A subset $I \subseteq T$ is called \emph{left (\textnormal{resp.}\ right) quasi-regular} if every element of $I$ is left (resp.\ right) quasi-regular. As in \cite[Sec.\ 4, Exercise 4]{Lam_1991}, we have the following definition.

\begin{defn} 
The \emph{Jacobson radical} of a (not necessarily unital) ring $T$ is
\[ \msJ(T) \colonequals \{a \in T : Ta \text{ is left quasi-regular}\}.\]
\end{defn}

\begin{rem}
Note that $\msJ(T)$ is a two-sided ideal containing all left (resp.\ right) quasi-regular  ideals of $T$. In particular, if $T$ has a multiplicative identity, then $\msJ(T)$ is the intersection of all maximal left (resp.\ right) ideals of $T$, which agrees with the usual characterizations of the Jacobson radical of a unital ring.
\end{rem}

\begin{thm}\label{thm:Wedderburn} (Wedderburn-Malcev Theorem)
Let $T$ be a finite unital ring of characteristic $p$. Then, there exists an $\F_p$-subalgebra $S$ of $T$ such that $T = S \oplus \msJ(T)$, and $S \cong T/\msJ(T)$ as $\F_p$-algebras. The ring $S$ is semisimple: for some $t \ge 1$, $S \cong \prod_{i=1}^t M_{n_i}(\F_{p^{d_i}})$, where each integer $d_i,n_i \ge 1$.
\end{thm}

An immediate consequence of the definition of $R'$ is that a subset $I \subseteq R$ is a left ideal of $R$ if and only if it is a left $R'$-module.  When viewing left ideals in this light, the following classical result will be useful:

\begin{lem} (Nakayama's Lemma) \cite[Lem.\ 4.22]{Lam_1991}
 \label{lem:Nakayama}
 Let $T$ be a unital ring and $I$ be a left ideal of $T$. Then, $I \subseteq \msJ(T)$ if and only if for any finitely generated left $T$-module $M$ and $N$ a submodule of $M$, $M = N + IM$ implies that $M = N$.
\end{lem}

To use Nakayama's Lemma, we need to know about the Jacobson radical of $R'$.  As it turns out, it is simply the Jacobson radical of $R$.

\begin{lem}
 \label{lem:radDorroh}
$\msJ(R') = \{0\} \times \msJ(R)$.
\end{lem}

\begin{proof}
The following proof is adapted from and analogous to that of \cite{rschwieb_2016}.  First, since $\{0\} \times R$ is a maximal ideal of $R'$ by Lemma \ref{lem:idealsofDorroh}, $\msJ(R') \subseteq \{0\} \times R$.

Let $(0,x) \in \{0\} \times \msJ(R)$.  By Lemma \ref{lem:idealsofDorroh}, $\{0\} \times \msJ(R)$ is a two-sided ideal of $R'$.  For any $(n,r) \in R'$, we have $(n,r)(0,x) = (0, nx + rx) \in \{0\} \times \msJ(R)$.  Since $nx + rx \in \msJ(R)$, it is left quasi-regular in $R$ and hence left quasi-regular in $R'$. Indeed, if $y$ is the left inverse of $nx + rx$ in $(R, \circ)$, then $(0,y)$ is the left inverse of $(n,r)(0,x)$ in $(R', \circ)$.  Thus, $\{0\} \times \msJ(R) \subseteq \msJ(R')$.

Now, let $(0, x) \in \msJ(R')$. Then, for every $(n,r) \in R'$, there exists $(m,y) \in R'$ such that $(n,r)(0,x) + (m,y) - (m,y)(n,r)(0,x) = (0,0)$. We have
\begin{align*}
(m,y) &= (m,y)(n,r)(0,x) - (n,r)(0,x) \\
&= (0, mnx + mrx + nyx + yrx - nx - rx) \in \{0\} \times R.
\end{align*}
This gives $m = 0$. So, for any $r \in R$, we may choose $n  = 0$ and then $(0,y) \circ (0,r)(0,x)=0$ in $R'$. This implies $rx + y - yrx = 0$ in $R$, i.e., $rx$ is left quasi-regular and $x \in \msJ(R)$.  The result follows.
\end{proof}

We close this section by demonstrating a form of Theorem \ref{thm:Wedderburn} that holds for finite nonunital rings in characteristic $p$.

\begin{prop}\label{prop:R=S+J}
Let $R$ be a finite (not necessarily unital) ring of characteristic $p$ and Jacobson radical $J$. Then, there exists a subring $S$ of $R$ such that $R = S \oplus J$ and $S$ is either $\{0\}$ or semisimple.
\end{prop}
\begin{proof}
By Theorem \ref{thm:Wedderburn} and Lemma \ref{lem:radDorroh}, there exists a semisimple subalgebra $S'$ of $R'$ such that
 \[R' = S' \oplus \msJ(R') = S' \oplus (\{0\} \times J)\]
 and
 \[ S' \cong R'/\msJ(R') = R'/(\{0\} \times J).\]
 Since $\{0\} \times R$ is an ideal of $R'$ by Lemma \ref{lem:idealsofDorroh}, we have
 \[R/J \cong (\{0\} \times R)/(\{0\} \times J),\]
so $R/J$ is isomorphic to an ideal of the semisimple ring $S'$. By \cite[Thm.\ VIII.5]{McDonald_1974}, $R/J$ is either $\{0\}$ or semisimple.  Since $S' \cap (\{0\} \times J) = \{0_{R'}\}$, there exists an ideal $\{0\} \times S$ of $S'$ such that
 \[\{0\} \times R = (\{0\} \times S) \oplus \msJ(R') = (\{0\} \times  S) \oplus (\{0\} \times J),\]
 i.e., $R = S \oplus J$, where $S$ is either $\{0\}$ or semisimple.
\end{proof}

\begin{rem}
Note that while $S$ (when nonzero) is isomorphic to a semisimple---hence unital---ring, we cannot immediately conclude that the action of $S$ on $J$ is nontrivial. That is, in the decomposition of Proposition \ref{prop:R=S+J}, it is possible that $SJ = JS =\{0\}$. This differs from the behavior of $S$ and $J$ in the decomposition $T=S \oplus J$ of Theorem \ref{thm:Wedderburn}, in which $1_S = 1_T$ and $SJ =JS = J$.
\end{rem}

\section{Basic properties and reduction theorems} \label{sec: reduction}

In this section, we establish some results for minimal covers by left and two-sided ideals, for the structure of $\eta_\ell$-elementary and $\eta$-elementary rings. Note that any statement regarding covers by left ideals has a natural analogue for right ideals. After proving Theorem \ref{thm:two-sided}, which classifies all $\eta$-elementary rings, we focus solely on $\eta_\ell$-elementary rings.

We begin with a variation on \cite[Thm.\ 2.2]{Werner_2015}, which gives a sufficient condition for the covering number of a direct product of rings to be equal to the covering number of one of the direct factors.

\begin{lem}\label{lem: direct products of rings}
Let $R$ be a finite ring such that $R = \prod_{i=1}^t R_i$, where each $R_i$ is a ring. Assume that $*$-ideals of $R$ respect this decomposition. That is, for each $*$-ideal $L \subseteq R$, there exist $*$-ideals $L_i \subseteq R_i$ such that $L = \prod_{i=1}^t L_i$.
\begin{enumerate}[(1)]
\item Each maximal $*$-ideal $M$ of $R$ has the form
\begin{equation*}
M = R_1 \times \cdots \times R_{i-1} \times M_i \times R_{i+1} \times \cdots \times R_t
\end{equation*}
for some $1 \le i \le t$, where $M_i$ is a maximal $*$-ideal of $R_i$.

\item $\eta_*(R) = \min_{1 \le i \le t}(\eta_*(R_i))$.

\item If $R$ is $\eta_*$-elementary, then $t=1$.
\end{enumerate}
\end{lem}
\begin{proof} 
(1) Let $M$ be a maximal $*$-ideal of $R$. By assumption, $M = \prod_{i=1}^t L_i$, where each $L_i$ is a $*$-ideal of $R_i$. If there exist two indices $j \ne k$ such that $L_j \ne R_j$ and $L_k \ne R_k$, then let $M' = \prod_{i=1}^t L_i'$, where $L_j' = R_j$ and $L_i'=L_i$ for $i \ne j$. We have $M \subsetneqq M' \subsetneqq R$, which contradicts the maximality of $M$. The result follows.

(2) This is identical to the proof of \cite[Thm.\ 2.2]{Werner_2015} with each use of ``maximal subring'' replaced with ``maximal $*$-ideal''.

(3) Suppose that $R$ is $\eta_*$-elementary, but $t \ge 2$. Then, each direct factor $R_i$ occurs as a residue ring $R/I$ of $R$ for some nonzero two-sided ideal $I$ of $R$. Since $R$ is $\eta_*$-elementary, $\eta_*(R) < \eta_*(R_i)$ for each $i$. This contradicts (2).
\end{proof}

\begin{lem}
 \label{lem:idealcoverrelations}
 Let $R$ be an $\eta_*$-elementary ring with minimal cover $L_1, \dots, L_m$ by $*$-ideals. If $I$ is a two-sided ideal of $R$ that is contained in $L_i$ for each $1 \le i \le m$, then $I = \{0\}$.
\end{lem}

\begin{proof}
Assume that $I$ is contained in $L_i$ for each $i$. Then, each $L_i/(L_i \cap I) = L_i/I$ is a proper $*$-ideal of $R/I$, and $L_1/I, \dots, L_m/I$ form a cover of $R/I$. If $I \neq \{0\}$, then $\eta_*(R) = m \ge \eta_*(R/I)$, a contradiction to $R$ being $\eta_*$-elementary.
\end{proof}

\begin{lem}\label{lem:Rcharp}\mbox{}
\begin{enumerate}[(1)]
\item If $R$ admits a finite cover by $*$-ideals, then $R$ contains a two-sided ideal $I$ of finite index such that $\eta_*(R) = \eta_*(R/I)$.
\item If $R$ is an $\eta_*$-elementary ring, then $R$ is finite and has prime characteristic.
\end{enumerate}
\end{lem}

\begin{proof}
(1) Suppose that $\eta_*(R) < \infty$ and $R=\bigcup_{i=1}^{\eta_*(R)} L_i$ is a minimal cover by $*$-ideals. By \cite[Lem.\ 4.1, 4.4]{Neumann_1954}, each $(L_i,+)$ has finite index in $(R,+)$, and by \cite[Lem.\ 1]{Lewin_1967}, their intersection contains a two-sided ideal $I$ of finite index in $R$, so $R/I$ is finite. Since the images of $L_i$ mod $I$ form a $*$-ideal cover of $R/I$, and conversely any $*$-ideal cover of $R/I$ can be lifted to a $*$-ideal cover of $R$, we have $\eta_*(R)=\eta_*(R/I)$. 

(2) Assume that $R$ is $\eta_*$-elementary. By (1), $R$ must be finite, because the condition $\eta_*(R)=\eta_*(R/I)$ forces $I=\{0\}$. It is well known that any finite (not necessarily unital) ring $R$ is isomorphic to a direct product of rings of prime power order, so $R=\prod_{i=1}^t R_i$ where $|R_i|=p_i^{d_i}$ for distinct primes $p_i$. It follows from Lemma \ref{lem: direct products of rings}(3) that $R$ has prime power order and hence has characteristic $p^d$ for some prime $p$ and integer $d \ge 1$.  We now argue as in the proof of \cite[Lem.\ 3.2(2), Prop.\ 3.3]{Swartz_Werner_2021}: $pR$ is a two-sided ideal of $R$, and, if $M$ is a maximal $*$-ideal of $R$, $M + pR$ is also a $*$-ideal.  Since $M \subseteq M + pR \subseteq R$, by the maximality of $M$, either $M + pR = M$ or $M + pR = R$.  However, by \cite[Lem.\ 3.2(1)]{Swartz_Werner_2021} (which holds for nonunital rings), $M + pR = R$ implies $M = R$, contradicting our assumptions.  Thus, $pR \subseteq M$ for all maximal $*$-ideals $M$.  In particular, if $L_1, \dots, L_m$ is a minimal cover of $R$ by $*$-ideals, without loss of generality we may assume each $L_i$ is a maximal $*$-ideal, and hence $pR \subseteq L_i$ for each $i$.  By Lemma \ref{lem:idealcoverrelations}, $pR = \{0\}$, and so $R$ has characteristic $p$.
\end{proof}

We can now begin to derive structural results on $R$.

\begin{lem}
 \label{lem:Jnot0}
 If $R$ is an $\eta_*$-elementary ring, then $\msJ(R) \neq \{0\}$.
\end{lem}

\begin{proof}
By Lemma \ref{lem:Rcharp}, $R$ has characteristic $p$. Suppose $\msJ(R) = \{0\}$.  Then, $\msJ(R') = \{0\}$ by Lemma \ref{lem:radDorroh}, and so $R'$ is semisimple by Theorem \ref{thm:Wedderburn}.  By Lemma \ref{lem:idealsofDorroh}, $\{0\} \times R$ is a two-sided ideal of $R'$, so $R$ is a direct sum of simple direct factors of $R'$; see e.g. \cite[Thm.\ VIII.5]{McDonald_1974}.  However, this implies that $R$ contains a multiplicative identity, which is a contradiction to $R$ having a cover by $*$-ideals.
\end{proof}

Henceforth, we will define $J \colonequals \msJ(R)$.  We can now apply Nakayama's Lemma (Lemma \ref{lem:Nakayama}) to obtain additional restrictions on $J$.

\begin{lem}\label{lem:JR=0}\mbox{}
\begin{enumerate}[(1)]
\item If $R$ is an $\eta_\ell$-elementary ring, then $JR = \{0\}$.
\item If $R$ is an $\eta$-elementary ring, then $JR = RJ = \{0\}$.
\end{enumerate}
\end{lem}

\begin{proof}
(1) Assume that $R$ is $\eta_\ell$-elementary. Clearly, $JR$ is a two-sided ideal of $R$.  Let $N$ be a maximal left ideal of $R$.  Then, $N + JR$ is also a left ideal of $R$, so either $N + JR = N$ or $N + JR = R$.  Viewing the left ideals of $R$ as left $R'$-modules and recalling that $\msJ(R') = \{0\} \times J$ by Lemma \ref{lem:radDorroh}, $N + JR = R$ in $R$ implies $N + \msJ(R')R = R$. By Nakayama's Lemma, $N = R$, a contradiction to the choice of $N$.  Thus, every maximal left ideal of $R$ contains $JR$.  By Lemma \ref{lem:idealcoverrelations}, $JR = \{0\}$.

(2) Assume that $R$ is $\eta$-elementary. Since every two-sided ideal of $R$ is a left ideal, proceeding as in (1) shows that $JR = \{0\}$. Similarly, working with right ideals and using the analogous form of Nakayama's Lemma for right modules, we see that $RJ = \{0\}$.
\end{proof}

We now have everything necessary to classify $\eta$-elementary rings.

\begin{proof}[Proof of Theorem \ref{thm:two-sided}]
$(\Leftarrow)$ Assume that $R^2=\{0\}$ and $(R,+) \cong C_p \times C_p$. Because $R^2=\{0\}$, the two-sided ideals of $R$ are precisely the additive subgroups of $(R,+)$. The group $C_p \times C_p$ has covering number $p+1$, so $\eta(R) = \sigma(C_p \times C_p) = p+1$. Also, each quotient group of $C_p \times C_p$ by a nontrivial subgroup is cyclic, hence is not coverable. It follows that $R$ is $\eta$-elementary.

$(\Rightarrow)$ Assume that $R$ is $\eta$-elementary. By Lemma \ref{lem:Rcharp}, $R$ is finite of characteristic $p$. Write $R = S \oplus J$ as in Proposition \ref{prop:R=S+J}. By Lemma \ref{lem:JR=0}, $JR = RJ = \{0\}$, so multiplication in $R$ follows the rule
\begin{equation}\label{eq: mult rule}
(s_1+x_1)(s_2+x_2) = s_1s_2,
\end{equation}
for all $s_1, s_2 \in S$ and $x_1, x_2 \in J$. From \eqref{eq: mult rule}, we see that $S$ is a two-sided ideal of $R$. Moreover, $J$ is nonzero by Lemma \ref{lem:Jnot0}; $J^2 = \{0\}$; and $I \subseteq J$ is a two-sided ideal of $R$ if and only if $(I,+)$ is a subgroup of $(J,+)$.

We claim that $S =\{0\}$. Suppose not, and let $e = 1_S$. If $(J,+) = \langle x \rangle$ is cyclic, then $R$ is not coverable. Indeed, by \eqref{eq: mult rule} we have $e(e+x) = e^2=e$, so any ideal of $R$ containing $e+x$ contains both $e$ and $x$, and hence must equal $R$. So, $(J,+)$ is not cyclic, and hence $J$ is coverable by two-sided ideals. Since ideal covers of $J$ correspond to subgroup covers of $(J,+)$, we have $\eta(J) = p+1$ by \cite[Thm.\ 2]{Cohn_1994}. But then, $\eta(R) < \eta(R/S) = \eta(J)=p+1$ because $R$ is $\eta$-elementary. This is impossible, because an ideal cover of $R$ induces an additive cover of $(R,+)$, and $(R,+)$ has covering number $p+1$ by \cite[Thm.\ 2]{Cohn_1994}. We conclude that $S = \{0\}$ and $R=J$. As a consequence, we have $R^2=\{0\}$.

It remains to show that $(R,+) \cong C_p \times  C_p$. As in the previous paragraph, $(R,+)$ cannot be cyclic, because then $R$ is not coverable. Since $|R|=p^d$ for some $d \ge 1$, $(R,+)$ contains a subgroup $H$ such that (as additive groups) $R/H \cong C_p \times C_p$. Since $R^2=\{0\}$, $H$ is a two-sided ideal of $R$. We have $\eta(R/H) = p+1 \le \eta(R) \le \eta(R/H)$, so $\eta(R)=\eta(R/H)$. Since $R$ is $\eta$-elementary, $H=\{0\}$, and we are done.
\end{proof}

For the remainder of the paper, we will focus exclusively on covers by left ideals and $\eta_\ell$-elementary rings. When $R$ is $\eta_\ell$-elementary, by Lemma \ref{lem:Rcharp}, $R$ is finite of characteristic $p$; $R = S \oplus J$ by Proposition \ref{prop:R=S+J}; and $JR = \{0\}$ by Lemma \ref{lem:JR=0}. If $S \neq \{0\}$, then $S$ is semisimple of characteristic $p$ and has a multiplicative identity $e$. Let $SJ \colonequals \{sx : s \in S \text{ and } x \in J\}$, which will be a left $S$-module if $S$ is nontrivial, and define $K \colonequals \{x \in J : Rx = \{0\}\}$.

\begin{lem}\label{lem: direct sums and K}
Let $R$ be an $\eta_\ell$-elementary ring.
\begin{enumerate}[(1)]
\item Both $SJ$ and $K$ are two-sided ideals of $R$, and $J = SJ \oplus K$.
\item $R = S \oplus SJ \oplus K$.
\item $S \oplus SJ$ is a two-sided ideal of $R$.
\item For every left ideal $L \subseteq R$, we have $L = \big(L \cap (S \oplus SJ)\big) \oplus (L \cap K)$, where $L \cap (S \oplus SJ)$ and $L \cap K$ are left ideals of the rings $S \oplus SJ$ and $K$, respectively.
\end{enumerate}
\end{lem}
\begin{proof}
(1) By Lemma \ref{lem:JR=0}, $JR=\{0\}$. With this is mind, it is routine to check that $SJ$ and $K$ are two-sided ideals of $R$. If $S = \{0\}$, then $R = J$, which implies $R^2 = J^2 = \{0\}$ by Lemma \ref{lem:JR=0}. Hence, $R = J = K$, and the result follows. So, assume $S \neq \{0\}$.  Both $SJ$ and $K$ are subsets of $J$, so $SJ + K \subseteq J$. For the reverse containment, let $y \in J$. Then, $ey \in SJ$ and $y-ey \in K$. Thus, $y \in SJ + K$ and $J = SJ + K$. To show that the sum is direct, let $x \in SJ \cap K$. Then, $ex = x$ because $x \in SJ$, and $ex = 0$ because $x \in K$. Hence, $x=0$ and $J = SJ \oplus K$.

(2) This follows from (1) and Proposition \ref{prop:R=S+J}.

(3) Consider the multiplication of two elements of $R$. Let $s_1, s_2 \in S$, $x_1, x_2 \in SJ$, and $k_1, k_2 \in K$. Since $JR = \{0\}$, we have $(s_1+x_1+k_1)(s_2+x_2+k_2) = s_1(s_2+x_2+k_2)$. Moreover, $s_1 k_2 = 0$. Thus,
\begin{equation*}
(s_1+x_1+k_1)(s_2+x_2+k_2) = s_1(s_2 + x_2) = (s_1 + x_1)(s_2 + x_2).
\end{equation*}
From this, it follows that $S \oplus SJ$ is closed under both left and right multiplication by elements of $R$, and hence is a two-sided ideal of $R$.

(4) Let $L$ be a left ideal of $R$ and let $L' = (L \cap (S \oplus SJ)) + (L \cap K)$. The containment $L' \subseteq L$ is clear, so let $a \in L$. By (2), we may write $a = b + k$, where $b \in S \oplus SJ$ and $k \in K$. Then, $b = eb = e(b+k) \in L$ and $k=a-b \in L$. Thus, $L = L'$. Lastly, $L$ is the direct sum of $L \cap (S \oplus SJ)$ and $L \cap K$ because $(S \oplus SJ) \cap K = \{0\}$.
\end{proof}

\begin{prop}\label{prop: K must be 0}
Let $R$ be an $\eta_\ell$-elementary ring. Then either $S \oplus SJ = \{0\}$ or $K = \{0\}$.  In particular, either $J = K$ or $J = SJ$.
\end{prop}
\begin{proof}
Let $R_1 = S \oplus SJ$ and $R_2 = K$. As additive groups, $R = R_1 \oplus R_2$ by Lemma \ref{lem: direct sums and K}(2). However, $R_1 R_2 = \{0\} = R_2 R_1$, so $R_1 \oplus R_2 \cong R_1 \times R_2$ as rings. By Lemma \ref{lem: direct sums and K}(4), left ideals of $R$ respect this decomposition. Thus, by Lemma \ref{lem: direct products of rings}(3), either $S \oplus SJ = \{0\}$ or $K=\{0\}$.
\end{proof}

\begin{prop}
 \label{prop:structR}
 Let $R$ be a noncommutative $\eta_\ell$-elementary ring. Then, $R = S \oplus J$, where $J = SJ$.  In particular,
 \[ R \cong \left\{\begin{pmatrix}
                    s & j \\
                    0 & 0
                   \end{pmatrix} : s \in S \text{ and } j \in J \right\},
\]
where the operations are matrix addition and matrix multiplication.
\end{prop}

\begin{proof}
By Lemma \ref{lem: direct sums and K}(2) and Proposition \ref{prop: K must be 0}, $R = S \oplus SJ \oplus K$ and either $S \oplus SJ = \{0\}$ or $K = \{0\}$. If $S \oplus SJ = \{0\}$, then $R = K$ and $R^2 = \{0\}$, which contradicts $R$ being noncommutative.  Hence, $S \oplus SJ \neq \{0\}$, which means $K = \{0\}$ by Proposition \ref{prop: K must be 0}, and hence $R = S \oplus J$, where $J = SJ$.  Finally, any $r_1, r_2 \in R$ can be expressed uniquely as $r_1 = s_1 + j_1$, $r_2 = s_2 + j_2$, where $s_1, s_2 \in S$, $j_1, j_2 \in J$, and, since
 \[ r_1 + r_2 = (s_1 + s_2) + (j_1 + j_2), \; r_1r_2 = (s_1 + j_1)(s_2 + j_2) = s_1s_2 + s_1j_2,\]
 the map $r = s + j \mapsto \left(\begin{smallmatrix} s & j \\ 0 & 0 \end{smallmatrix}\right)$
is an isomorphism.
\end{proof}

The rings described Theorem \ref{thm:main}(1) have the form specified by Proposition \ref{prop:structR}. To prove Theorem \ref{thm:main}(1), it remains to show two things. First, that there exist $\eta_\ell$-elementary rings of the form given in Proposition \ref{prop:structR}. Second, when $R=S \oplus J$ is $\eta_\ell$-elementary, we must prove that $S$ is a simple ring and $J$ is a simple left $S$-module. These claims are dealt with in Sections \ref{sec: new numbers} and \ref{sec: structure}.

\section{An infinite family of $\eta_\ell$-elementary rings} \label{sec: new numbers}

Let $q$ be a prime power and $n \ge 1$ an integer. We now construct examples of nonunital rings $R$ which are the union of finitely many proper left ideals, show that they are $\eta_\ell$-elementary, and prove that $\eta_\ell(R) = q^n + q^{n-1} + \cdots + q + 1 = \frac{q^{n+1}-1}{q-1}$. 

Throughout this section, let $R$ be the set consisting of all $(n+1) \times (n+1)$ matrices with entries from $\F_q$ whose last row is zero, which is the ring described in Theorem \ref{thm:main}(1). Then,

\begin{itemize}
\item $R$ is a ring under matrix addition and multiplication. 
\item Elements of $R$ can be represented as $(A|v)$, where $A \in M_n(\F_q)$ and $v \in \F_q^n$ is a column vector. Multiplication in $R$ follows the rule $$(A|v)(B|w) = (AB|Aw).$$
\item Under the representation above, let $e=(I_n|0)$, where $I_n$ is the $n \times n$ identity matrix. Then, $e$ is a left identity of $R$, but is not a right identity. We have $eRe \cong M_n(\F_q)$ as rings. Let $S := eRe$.
\end{itemize}

\begin{lem}\label{lem: ideals and modules}
Let $L \subseteq R$. Then, $L$ is a left ideal of $R$ if and only if $L$ is a left $S$-submodule of $R$.
\end{lem}
\begin{proof}
In both implications, closure under addition is trivial.

$(\Rightarrow)$ This is clear, because $S \subseteq R$ and $RL \subseteq L$.

$(\Leftarrow)$ Assume that $L$ is a left $S$-submodule of $R$. Let $r \in R$ and $\ell \in L$. Since $e$ is a left identity of $R$, $r\ell = (er)(e\ell) = (ere)\ell$, and this is in $L$ because $ere \in S$.
\end{proof}

The following facts about $S$-modules hold because $S$ is a (semi)simple ring.
\begin{itemize}
\item $S$-modules are semisimple. Hence, a nonzero $S$-module is a direct sum of simple $S$-modules.


\item Since $S \cong M_n(\F_q)$, simple $S$-modules are isomorphic (as $S$-modules) to $\F_q^n$. In particular, all such simple modules have size $q^n$. Consequently, any finite $S$-module must have order $q^{nk}$ for some $k \ge 0$.

\item $S$-modules are injective. So, if $L$ and $M$ are $S$-modules and $L \subseteq M$, then there exists an $S$-module $L'$ such that $M = L \oplus L'$. 
\end{itemize}

\begin{lem}\label{lem: ideals of size q^{n^2}}
Let $L$ be a left ideal of $R$. Then, $L$ is a maximal left ideal of $R$ if and only if $|L| = q^{n^2}$.
\end{lem}
\begin{proof}
By Lemma \ref{lem: ideals and modules}, $L$ is an $S$-module. Since $L$ is an injective $S$-module, there exists an $S$-module $L'$ such that $R = L \oplus L'$. The $S$-module $L'$ is itself semisimple, so $L$ is maximal if and only if $L'$ is simple, if and only if $|L|=q^{n^2}$.
\end{proof}

Next, we define two families of left ideals in $R$. Both families will turn out to consist of maximal left ideals that are also cyclic left ideals of $R$.

\begin{defn}\label{def: L_v}
For each $v \in \F_q^n$ (viewed as a column vector) and 1-dimensional subspace $V \subseteq \F_q^n$, let $L_v := \{(A|Av) : A \in S\}$ and $N_V := \{(A|w) : AV = 0 \text{ and } w \in \F_q^n\}$. That is, $N_V$ consists of the elements $(A|w) \in R$ such that $V$ is contained in the right nullspace of $A$, and there are no restrictions on $w$.
\end{defn}

\begin{lem}\label{lem: L_v lemma}\mbox{}
\begin{enumerate}[(1)]
\item For each $v \in \F_q^n$, $L_v$ is a maximal left ideal of $R$ and is generated by $(I_n|v)$.
\item For $v, w \in \F_q^n$, $L_v = L_w$ if and only if $v=w$.
\item Let $(A|v) \in R$, where $A \in S$ is invertible. Then, $(A|v) \in L_w$, where $w = A^{-1}v$.
\end{enumerate}
\end{lem}
\begin{proof}
(1) It is straightforward to check that $L_v$ is a left ideal of $R$. Since $(A|Av) = (A|0)(I_n|v)$, the left ideal $L_v$ is generated by $(I_n|v)$. Finally, $|L_v| = |M_n(\F_q)| = q^{n^2}$, so $L_v$ is maximal by Lemma \ref{lem: ideals of size q^{n^2}}.

(2) This follows from (1).

(3) Clearly, $(A|v) = (A|0)(I_n|A^{-1}v) \in L_w$.
\end{proof}

\begin{lem}\label{lem: N_V lemma}\mbox{}
\begin{enumerate}[(1)]
\item For each 1-dimensional subspace $V \subseteq \F_q^n$, $N_V$ is a maximal left ideal of $R$.
\item For each 1-dimensional subspace $V \subseteq \F_q^n$, $N_V$ is a cyclic left ideal of $R$.
\item Let $V$ and $W$ be 1-dimensional subspaces of $\F_q^n$. Then, $N_V = N_W$ if and only if $V=W$.
\item Let $(A|v) \in R$ such that $A \in S$ is not invertible. Then, $(A|v) \in N_V$ for any 1-dimensional subspace $V \subseteq \F_q^n$ such that $AV = 0$.
\end{enumerate}
\end{lem}
\begin{proof}
(1) Let $V$ be a 1-dimensional subspace of $\F_q^n$. It is easy to verify that $N_V$ is a left ideal of $R$. For maximality, let $V^\perp$ be the subspace complement to $V$ in $\F_q^n$. Then, $\dim_{\F_q}(V^\perp) = n-1$, so $|V^\perp| = q^{n-1}$. Now, $(A|w) \in N_V$  if and only if each row of $A$ is the transpose of a vector in $V^\perp$. Thus, there are $q^{(n-1)n}$ choices for $A$ and $q^n$ choices for $w$, so $|N_V| = q^{(n-1)n} q^n = q^{n^2}$. Hence, $N_V$ is maximal by Lemma \ref{lem: ideals of size q^{n^2}}.

(2) Fix a 1-dimensional subspace $V \subseteq \F_q^n$ with complement subspace $V^\perp$. Let $\{b_1, \ldots, b_{n-1}\}$ be a basis (of column vectors) for $V^\perp$. Let
\begin{equation*}
B = \begin{augmat}{1} b_1^T & 0 \\ \vdots & \vdots \\ b_{n-1}^T & 0 \\ 0 & 1 \end{augmat} \in R.
\end{equation*}
Then, $b_i^T v = 0$ for all $1 \le i \le n-1$ and all $v \in V$, so $B \in N_V$. Next, given $(A|w) \in N_V$, the transpose of row $i$ of $A$ is in $V^\perp$, so that row is an $\F_q$-linear combination of $b_1^T, \ldots, b_{n-1}^T$. For $1 \le i \le n$ and $1 \le j \le n-1$, let $c_{i,j} \in \F_q$ be such that row $i$ of $A$ is equal $c_{i,1}b_1^T + \cdots + c_{i,n-1}b_{n-1}^T$. Also, let $w_i$ be the entry of $w$ in row $i$. Then,
\begin{equation*}
(A|w) = \begin{augmat}{4} c_{1,1} & \cdots & c_{1,n-1} & w_1 & 0\\ \vdots & \ddots & \vdots & \vdots & \vdots \\ c_{n,1} & \cdots & c_{n,n-1} & w_n & 0\end{augmat} \begin{augmat}{1} b_1^T & 0 \\ \vdots & \vdots \\ b_{n-1}^T & 0 \\ 0 & 1 \end{augmat}.
\end{equation*}
This shows that $N_V = RB$.

(3) This follows from the observation that for any 1-dimensional subspaces $V$ and $W$, there exist matrices $A_V$ and $A_W$ such that the right nullspaces of $A_V$ and $A_W$ are exactly $V$ and $W$, respectively. 

(4) This is clear from the definition of $N_V$.
\end{proof}

\begin{prop}\label{prop: (n,1)_q covering numbers}
The left ideal covering number of $R$ is $\eta_\ell(R) = \frac{q^{n+1}-1}{q-1}$.
\end{prop}
\begin{proof}
By Lemmas \ref{lem: L_v lemma}(1), \ref{lem: N_V lemma}(1) and \ref{lem: N_V lemma}(2), each left ideal $L_v$ and $N_V$ is a maximal left ideal of $R$, and is generated (as a left ideal) by a single element. Thus, each $L_v$ and $N_V$ must be part of any cover of $R$ by left ideals. By Lemma \ref{lem: L_v lemma}(2), the number of distinct maximal left ideals $L_v$ is $|\F_q^n| = q^n$; and by Lemma \ref{lem: N_V lemma}(3), the number of distinct maximal left ideals $N_V$ is equal to the number of 1-dimensional subspaces of $\F_q^n$, which is (using Gaussian binomial coefficients) $\binom{n}{1}_q = 1 + q + \cdots + q^{n-1}$. So, $\eta_\ell(R) \ge q^n + q^{n-1} + \cdots + 1$. On the other hand, $\eta_\ell(R) \le q^n + q^{n-1} + \cdots + 1$ because the collection of all left ideals $L_v$ and $N_V$ forms a cover of $R$ by Lemmas \ref{lem: L_v lemma}(3) and \ref{lem: N_V lemma}(4). We conclude that $\eta_\ell(R) = q^n + q^{n-1} + \cdots + 1 = \frac{q^{n+1}-1}{q-1}$.
\end{proof}

It remains to demonstrate that $R$ is $\eta_\ell$-elementary. For this, we will describe all of the two-sided ideals of $R$.

\begin{lem}\label{lem: two-sided ideals of R}
The Jacobson radical of $R$ is $\msJ(R) = \{(0|v) : v \in \F_q^n\}$, and this is the only proper nonzero two-sided ideal of $R$.
\end{lem}
\begin{proof}
Let $J=\{(0|v) : v \in \F_q^n\}$. Then, $J$ is a proper nonzero subset of $R$, and it is easy to see that $J$ is a two-sided ideal of $R$. We will show that $J$ is the only such ideal with these properties. Let $I$ be any proper nonzero two-sided ideal of $R$. 

Suppose that there exists $(A|0) \in I$ with $A \ne 0$. Then, $I$ contains the two-sided ideal of $M_n(\F_q)$ generated by $A$, which is equal to $M_n(\F_q)$ itself. Hence, the left identity $e$ of $R$ is in $I$, and since $IR \subseteq I$, we see that $I=R$. This is a contradiction. So, whenever $(A|0) \in I$, $A$ must be 0.

Now, let $(B|v) \in I$. Then, $(B|0) = (B|v)e \in I$. By the previous paragraph, $B = 0$, so each element of $I$ has the form $(0|v)$. Thus, $I \subseteq J$. Since $I$ is nonzero and $M_n(\F_q)$ acts irreducibly on $\F_q^n$, we must have $I=J$.

Finally, $\msJ(R)$ is a two-sided ideal of $R$. This ideal is nonzero because $J \subseteq \msJ(R)$, and it is proper because $(I_n|0) \notin \msJ(R)$. Thus, $\msJ(R)=J$.
\end{proof}

\begin{cor}\label{cor: eta-elementary matrix ring}
The ring $R$ is $\eta_\ell$-elementary.
\end{cor}
\begin{proof}
By Proposition \ref{prop: (n,1)_q covering numbers}, $R$ is coverable by left ideals. By Lemma \ref{lem: two-sided ideals of R}, $\msJ(R)$ is the only proper nonzero two-sided ideal of $R$. The residue ring $R/\msJ(R) \cong M_n(\F_q)$ is unital, and hence not coverable by left ideals. Therefore, $R$ is $\eta_\ell$-elementary.
\end{proof}

\begin{rem}
Because $R$ contains a left identity, $R$ cannot be covered by right ideals or two-sided ideals. Hence, $\eta_r(R) = \eta(R) = \infty$.
\end{rem}

\section{Structure of noncommutative $\eta_\ell$-elementary rings} \label{sec: structure}

Throughout this section, let $R$ be a noncommutative $\eta_\ell$-elementary ring. By Proposition \ref{prop:structR}, $R=S \oplus J$, where $S$ is a subring of $R$ which is semisimple of characteristic $p$ and $J=\msJ(R)$ is a left $S$-module. Furthermore, $JR=\{0\}$ and
\[ R \cong \left\{\begin{pmatrix}
                    s & j \\
                    0 & 0
                   \end{pmatrix} : s \in S \text{ and } j \in J \right\},
\]
where the operations are usual matrix addition and multiplication. 

The purpose of this section is to prove that $S$ is a simple ring and $J$ is a simple left $S$-module. Consequently, we can deduce that $S \cong M_n(\F_q)$ for some prime power $q$ and integer $n \ge 1$, and that $J \cong \F_q^n$. This then corresponds to one of the $\eta_\ell$-elementary rings constructed in Section \ref{sec: new numbers}, which will complete the proof of Theorem \ref{thm:main}.

\begin{prop}\label{prop: S must be simple}
 $S$ is a simple ring. Thus, $S \cong M_n(\F_q)$ for some prime power $q$ and integer $n \ge 1$.
\end{prop}
\begin{proof}
Since $S$ is semisimple, we have $S = \bigoplus_{i=1}^t S_i$ for simple unital rings $S_1, \ldots, S_t$ that mutually annihilate one another (that is, $S_i S_j = \{0\}$ whenever $i \ne j$). We will show that $t=1$. For each $1 \le i \le t$, let $e_i$ be the multiplicative identity of $S_i$, so that $e=e_1 + \cdots + e_t$ and $e_i e_j = 0$ when $i \ne j$. Also, for each $i$ let $J_i = S_i J$ (it is possible that $J_i = \{0\}$). We have $J = SJ$ and
\begin{equation*}
R = S \oplus SJ = \Big(\bigoplus_{i=1}^t S_i\Big) \oplus \Big(\bigoplus_{i=1}^t S_i\Big)J = \Big(\bigoplus_{i=1}^t S_i\Big) \oplus \Big(\bigoplus_{i=1}^t J_i\Big) = \bigoplus_{i=1}^t (S_i \oplus J_i).
\end{equation*}
We know that $e_i e_j = 0$ when $i \ne j$, and $JR = \{0\}$, so the rings $S_i \oplus J_i$ and $S_j \oplus J_j$ mutually annihilate one another when $i \ne j$. It follows that, as rings, $R \cong \prod_{i=1}^t (S_i \oplus J_i)$. We will show that left ideals of $R$ respect this decomposition.

Let $L$ be a left ideal of $R$. For each $i$, let $L_i = L \cap (S_i \oplus J_i)$, which is a left ideal of $S_i \oplus J_i$. Certainly, $\bigoplus_{i=1}^t L_i \subseteq L$. For the reverse inclusion, assume that $a \in L$. Since $a \in R$, we may write $a = \sum_{i=1}^t a_i$, where each $a_i \in S_i \oplus J_i$. Then, $e_i a = e_i a_i = a_i \in S_i \oplus J_i$. But, $e_i a \in L$, so $a_i \in L_i$. Thus, $a = \sum_{i=1}^t a_i \in \bigoplus_{i=1}^t L_i$, and so $L = \bigoplus_{i=1}^t L_i$.

We can now apply Lemma \ref{lem: direct products of rings}(3) and conclude that $t=1$. Thus, $S \cong R/J$ is a simple unital ring.
\end{proof}

Next, suppose that $S \cong M_n(\F_q)$ for some prime power $q$ and integer $n \ge 1$. Let $e$ be the multiplicative identity of $S$, which is also a left identity of $R$. Since $S$ is simple, we may express $J$ as $J = \bigoplus_{i=1}^\lambda L_i$, where $\lambda \ge 1$ and each $L_i$ is a simple left $S$-module. So, $|L_i| = q^n$ for each $i$. Our goal is to show that $\lambda = 1$.

Note that since $JR = \{0\}$, any left ideal contained in $J$ is a two-sided ideal of $R$. In particular, $I = \bigoplus_{i=2}^\lambda L_i$ (which is $\{0\}$ if $\lambda=1$) is a two-sided ideal of $R$, and $R/I$ is isomorphic to the ring of Section \ref{sec: new numbers}. So, $R$ is coverable by left ideals and $\eta_\ell(R) \le \frac{q^{n+1}-1}{q-1} = \binom{n+1}{1}_q$.

\begin{defn}\label{def: S_x notations}
For each $x \in J$, let $S_x := \{A + Ax : A \in S\}$. When $\lambda \ge 2$, let $\widehat{J}_1 := \bigoplus_{i=2}^\lambda L_i$ and $\widehat{J}_2 := L_1 \oplus \bigoplus_{i=3}^\lambda L_i$. Finally, when $\lambda \ge 2$ and $x \in L_1$, define $T_{1,x} := S_x \oplus \widehat{J}_1$; and if $x \in L_2$, then define $T_{2,x} := S_x \oplus \widehat{J}_2$.
\end{defn}

We will show that each $T_{1,x}$ and each $T_{2,x}$ is a maximal left ideal of $R$ with a two-sided ideal complement in $R$. By the next lemma, each such left ideal must be part of any minimal cover of $R$.

\begin{lem}\label{lem: max ideal with complement}
Let $R$ be an $\eta_\ell$-elementary ring with minimal cover $I_1, \ldots, I_m$ by left ideals. If $R=M \oplus I$ for some maximal left ideal $M$ and two-sided ideal $I \neq \{0\}$ of $R$, then $M=I_k$ for some $k$.
\end{lem}
\begin{proof}
If $M \subseteq I_k$ for some $k$, then $M=I_k$ by maximality and we are done. So, suppose that $M \not\subseteq I_k$ for all $k$. Then, $I_1 \cap M,\ldots,I_m \cap M$ is a cover of $M$ by left ideals, so $\eta_\ell(R)=m \ge \eta_\ell(M) = \eta_\ell(R/I)$, a contradiction to $R$ being $\eta_\ell$-elementary.
\end{proof}

\begin{lem}\label{lem: S_x lemma} Assume that $\lambda \ge 2$.
\begin{enumerate}[(1)]
\item For each $x \in J$, $S_x$ is a left ideal of $R$, and $S_x \cong S$ as rings.
\item Let $x \in L_1 \cup L_2$. Then, $T_{1,x}$ and $T_{2,x}$ are both maximal left ideals of $R$. If $x \in L_1$, then $R = T_{1,x} \oplus L_1$; and if $x \in L_2$, then $R = T_{2,x} \oplus L_2$.
\item Each $T_{1,x}$ and $T_{2,x}$ is unique. That is, if $x \in L_1$ and $y \in L_2$, then $T_{1,x} \ne T_{2,y}$; if $x, y \in L_1$ and $x \ne y$, then $T_{1,x} \ne T_{1,y}$; and if $x, y \in L_2$ and $x \ne y$, then $T_{2,x} \ne T_{2,y}$.
\end{enumerate}
\end{lem}
\begin{proof}
(1) Let $x \in J$. It is clear that $S_x$ is closed under addition. Given $r \in R$, we may write $r = B+ y$, where $B \in S$ and $y \in J$. Then, for any $A+Ax \in S_x$, we have $r(A+Ax) = BA + BAx \in S_x$. Finally, one may verify that the map $S \to S_x$ sending $A \in S$ to $A+Ax \in S_x$ is a ring isomorphism.

(2) We prove this for $T_{1,x}$; the proof for $T_{2,x}$ is similar. Assume that $x \in L_1$. By (1), $S_x$ is a left ideal of $R$, so $T_{1,x}$ is a sum of left ideals, hence itself is a left ideal. For maximality, note that left ideals of $R$ correspond to left $S$-submodules of $R$, and any maximal left $S$-submodule of $R$ has order $q^{n(n+\lambda)-n} = q^{n(n+\lambda-1)} = |T_{1,x}|$. Lastly, we show that $T_{1,x} \cap L_1 = \{0\}$. If $y \in T_{1,x} \cap L_1$, then $y = A + Ax$ for some $A \in S$. So, $A = y-Ax \in S \cap J = \{0\}$, which means that $y=0$. Thus, $T_{1,x} \subsetneqq T_{1,x} \oplus L_1$, and by maximality $T_{1,x} \oplus L_1 = R$.

(3) When $x \in L_1$ and $y \in L_2$, the left ideal $T_{1,x}$ contains nonzero elements of $J_2$, while $T_{2,y}$ does not. Similarly, $T_{2,y}$ contains elements of $L_1$ that are not in $T_{1,x}$. So, $T_{1,x} \ne T_{2,y}$. In the case where $x, y \in L_1$, suppose that $T_{1,x} = T_{1,y}$. Recalling that $e$ is the identity element of $S$, we see that $e + x \in T_{1,x} = T_{1,y}$. So, $e+x = A+Ay$ for some $A \in S$. The condition $S \cap J = \{0\}$ forces $A=e$ and $x=y$. Similarly, for $x, y \in L_2$, having $T_{2,x} = T_{2,y}$ implies that $x=y$.
\end{proof}

\begin{prop}\label{prop: J must be simple}
$J$ is a simple left $S$-module. Thus, $J \cong \F_q^n$.
\end{prop}
\begin{proof}
Suppose that $\lambda  \ge 2$. As noted before Definition \ref{def: S_x notations}, $\eta_\ell(R) \le \binom{n+1}{1}_q$. Since $|L_1| = |L_2| = q^n$, by Lemma \ref{lem: S_x lemma}(3), $R$ contains $2q^n$ left ideals of the form $T_{1,x}$ or $T_{2,x}$. Recall that because $JR = \{0\}$, both $L_1$ and $L_2$ are two-sided ideals of $R$. So, by Lemma \ref{lem: S_x lemma}(2), each $T_{1,x}$ and each $T_{2,x}$ is a maximal left ideal of $R$ with a two-sided ideal complement in $R$. 

Since $R$ is $\eta_\ell$-elementary, by Lemma \ref{lem: max ideal with complement}, each $T_{1,x}$ and each $T_{2,x}$ must be part of any minimal cover of $R$ by left ideals. Thus, $\eta_\ell(R) \ge 2q^n$. This is a contradiction, because $\binom{n+1}{1}_q < 2q^n$. Therefore, $\lambda=1$ and the result follows.
\end{proof}

We complete the paper with our proof of Theorem \ref{thm:main}.

\begin{proof}[Proof of Theorem \ref{thm:main}]
 First, it suffices to prove the result for $\eta_\ell$-elementary rings, since the proof of (2) follows from passing to the opposite ring.  Let $R$ be a noncommutative $\eta_\ell$-elementary ring.  By Lemma \ref{lem:Rcharp}(2), $R$ is finite and has characteristic $p$. By Proposition \ref{prop:structR}, $R = S \oplus J$, where $J$ is the Jacobson radical of $R$, $S \neq \{0\}$ is semisimple, $SJ = J$, and by Lemma \ref{lem:JR=0}, $JR = \{0\}$.  By Proposition \ref{prop: S must be simple}, $S$ is simple, i.e., $S \cong M_n(\F_q)$ for some prime power $q$ and integer $n \ge 1$.  By Proposition \ref{prop: J must be simple}, $J$ is a simple left $S$-module, thus $J \cong \F_q^n$.  Finally, $R$ is actually $\eta_\ell$-elementary by Corollary \ref{cor: eta-elementary matrix ring}, and $\eta_\ell(R) = \frac{q^{n+1} - 1}{q - 1}$ by Proposition \ref{prop: (n,1)_q covering numbers}.
\end{proof}

\section{Covering numbers of modules}\label{sec: modules}

In this section, we prove Theorem \ref{thm: all module covers}, which determines all possible finite covering numbers of modules by submodules. Throughout, we work with left modules, although our proofs and techniques apply equally well to right modules.

\begin{notn}\label{not: sigma_R(M)}
For a unital ring $R$ and a left $R$-module $M$, we let $\sigma_R(M)$ be the cardinality of a minimal cover of $M$ by proper left $R$-submodules.
\end{notn}

\begin{lem}\label{lem: reduce to finite semisimple}
Let $R$ be a unital ring and let $M$ be a left $R$-module that admits a finite cover by proper left $R$-submodules. Then, there exist a residue ring $\olR$ of $R$ and a quotient module $\olM$ of $M$ such that $\olR$ is a finite semisimple ring, $\olM$ is a left $\olR$-module, and $\sigma_{R}(M) = \sigma_{\olR}(\olM)$.
\end{lem}
\begin{proof}
Following the proof of \cite[Thm.\ 2.2(2),(3)]{Khare_Tikaradze_2022}, we may assume that $M$ and $R$ are both finite. As in \cite{Khare_Tikaradze_2022}, we may assume that $M$ is finite by considering the induced cover of the underlying abelian group. Furthermore, we may view $M$ as a faithful left $R/{\rm{ann}}_R(M)$-module, but since $M$ is finite, this is a finite quotient. So, without loss of generality we may assume that $R$ is also finite.

Let $J = \msJ(R)$ be the Jacobson radical of $R$, let $\olR=R/J$, and let $\olM$ be the left $\olM$-module $M/JM$. Since $M$ is finitely generated, $JM$ equals the intersection of all maximal $R$-submodules of $M$; and so any minimal irredundant cover of $M$ (by maximal $R$-submodules) can be projected to a minimal cover of $\olM$ by $\olR$-submodules. Conversely, a minimal cover of $\olM$ can be lifted to a minimal cover of $M$. Thus, $\sigma_{R}(M)=\sigma_{\olR}(\olM)$.
\end{proof}

\begin{lem}\label{lem: reduce to finite simple}
Let $R$ be a finite semisimple ring and let $M$ be a left $R$-module that admits a finite cover by left $R$-submodules. Then, there exists a simple direct summand $S$ of $R$ and an $R$-submodule $N \subseteq M$ such that $N$ is a left semisimple $S$-module and $\sigma_{R}(M) = \sigma_{S}(N)$.
\end{lem}
\begin{proof}
Since $R$ is semisimple, we may write $R = \bigoplus_{i=1}^t S_i$ for some $t \ge 1$, where each $S_i$ is a finite simple ring. For each $i$, let $M_i = S_i  M$, so that $M = \bigoplus_{i=1}^t M_i$. If $M_i$ is a cyclic $S_i$-module for each $i$, then $M$ is not coverable. So, at least one $M_i$ must admit a cover by proper $S_i$-submodules. We may now apply \cite[Prop.\ 2.5]{Khare_Tikaradze_2022} (which is the analog of Lemma \ref{lem: direct products of rings} for modules) and conclude that there exists $1 \le j \le t$ such that $\sigma_{R}(M) = \sigma_{S_j}(M_j)$.
\end{proof}

\begin{proof}[Proof of Theorem \ref{thm: all module covers}]
By Lemmas \ref{lem: reduce to finite semisimple} and \ref{lem: reduce to finite simple}, we may assume that $R$ is a finite simple ring and $M$ is a finite left $R$-module. So, $R=M_n(\F_q)$ for some prime power $q$ and an integer $n \ge 1$. The $R$-module $M$ is semisimple, so we may write $M = \bigoplus_{i=1}^\lambda L_i$, where $\lambda \ge 1$ and each $L_i$ is a simple $R$-module. Now, each $L_i$ is isomorphic (as a left $R$-module) to $\F_q^n$. Treating the elements of each $L_i$ as column vectors, we may express $M$ as the set of all $n \times \lambda$ matrices with entries from $\F_q$, and treat the action of $R$ on $M$ as matrix multiplication.

If $1 \le \lambda \le n$, then $M$ is the cyclic $R$-module generated by the $n \times \lambda$ matrix where each diagonal entry is 1 and all other entries are 0. Since $M$ is coverable, we must have $\lambda \ge n+1$. In this case, we may view $M$ as a nonunital ring with a left identity such that the left ideals of $M$ are precisely the left $R$-submodules of $M$. Indeed, embed $M$ in the matrix ring $M_{\lambda}(\F_q)$ by sending each $m \in M$ to the $\lambda \times \lambda$ matrix $\binom{m}{0}$, where the final $\lambda - n$ rows are all zero. One may then check (cf.\ Lemma \ref{lem: ideals and modules}) that $L \subseteq M$ is a left ideal of $M$ if and only if $L$ is a left $R$-submodule of $M$. Thus, $\sigma_{R}(M) = \eta_\ell(M)$. Finally, the only $\eta_\ell$-elementary residue rings of $M$ are isomorphic to the ring of Theorem \ref{thm:main}(1). Therefore, $\sigma_R(M)=\tfrac{q^{n+1}-1}{q-1}$.
\end{proof} 

\section*{Acknowledgments}
The authors thank the referee for their feedback and careful reading of the paper.

\section*{Disclosure Statement}
The authors have no competing interests to declare.

\bibliographystyle{plainurl}
\bibliography{references-covering-rings-by-proper-ideals.bib}

@article{Chen_2025,
 author = {Chen, Malcolm Hoong Wai},
 title = {Rings as unions of proper ideals},
 year = {2025},
pages={6 p.},
 howpublished = {Preprint, {arXiv}:2508.05455 [math.{RA}] (2025)},
 url = {https://arxiv.org/abs/2508.05455},
 arXiv = {arXiv:2508.05455},
note = {Preprint, {arXiv}:2508.05455}
}

@article {Cohn_1994,
    AUTHOR = {Cohn, J. H. E.},
     TITLE = {On {$n$}-sum groups},
   JOURNAL = {Math. Scand.},
  FJOURNAL = {Mathematica Scandinavica},
    VOLUME = {75},
      YEAR = {1994},
    NUMBER = {1},
     PAGES = {44--58},
      ISSN = {0025-5521,1903-1807},
   MRCLASS = {20D60 (20E15)},

       DOI = {10.7146/math.scand.a-12501},
       URL = {https://doi.org/10.7146/math.scand.a-12501},
}

@article {Garonzi_Kappe_Swartz_2022,
    AUTHOR = {Garonzi, Martino and Kappe, Luise-Charlotte and Swartz, Eric},
     TITLE = {On integers that are covering numbers of groups},
   JOURNAL = {Exp. Math.},
  FJOURNAL = {Experimental Mathematics},
    VOLUME = {31},
      YEAR = {2022},
    NUMBER = {2},
     PAGES = {425--443},
      ISSN = {1058-6458,1944-950X},
   MRCLASS = {20D60 (20B15)},

       DOI = {10.1080/10586458.2019.1636425},
       URL = {https://doi.org/10.1080/10586458.2019.1636425},
}

@article {Khare_2009,
    AUTHOR = {Khare, Apoorva},
     TITLE = {Vector spaces as unions of proper subspaces},
   JOURNAL = {Linear Algebra Appl.},
  FJOURNAL = {Linear Algebra and its Applications},
    VOLUME = {431},
      YEAR = {2009},
    NUMBER = {9},
     PAGES = {1681--1686},
      ISSN = {0024-3795,1873-1856},
   MRCLASS = {15A03},

       DOI = {10.1016/j.laa.2009.06.001},
       URL = {https://doi.org/10.1016/j.laa.2009.06.001},
}

@article {Khare_Tikaradze_2022,
    AUTHOR = {Khare, Apoorva and Tikaradze, Akaki},
     TITLE = {Covering modules by proper submodules},
   JOURNAL = {Comm. Algebra},
  FJOURNAL = {Communications in Algebra},
    VOLUME = {50},
      YEAR = {2022},
    NUMBER = {2},
     PAGES = {498--507},
      ISSN = {0092-7872,1532-4125},
   MRCLASS = {13C05 (13F05 13F10 13L05)},

       DOI = {10.1080/00927872.2021.1959922},
       URL = {https://doi.org/10.1080/00927872.2021.1959922},
}

@book {Lam_1991,
    AUTHOR = {Lam, T. Y.},
     TITLE = {A first course in noncommutative rings},
    SERIES = {Graduate Texts in Mathematics},
    VOLUME = {131},
 PUBLISHER = {Springer-Verlag, New York},
      YEAR = {1991},
     PAGES = {xvi+397},
      ISBN = {0-387-97523-3},
   MRCLASS = {16-01},

       DOI = {10.1007/978-1-4684-0406-7},
       URL = {https://doi.org/10.1007/978-1-4684-0406-7},
}

@book {McDonald_1974,
    AUTHOR = {McDonald, Bernard R.},
     TITLE = {Finite rings with identity},
    SERIES = {Pure and Applied Mathematics},
    VOLUME = {28},
 PUBLISHER = {Marcel Dekker, Inc., New York},
      YEAR = {1974},
     PAGES = {ix+429},
   MRCLASS = {16A44 (13B05)},

}

@article {Parmenter_1994,
    AUTHOR = {Parmenter, M. M.},
     TITLE = {Finite coverings of rings by ideals},
   JOURNAL = {Canad. Math. Bull.},
  FJOURNAL = {Canadian Mathematical Bulletin. Bulletin Canadien de
              Math\'ematiques},
    VOLUME = {37},
      YEAR = {1994},
    NUMBER = {1},
     PAGES = {97--99},
      ISSN = {0008-4395,1496-4287},
   MRCLASS = {16D25},

       DOI = {10.4153/CMB-1994-015-5},
       URL = {https://doi.org/10.4153/CMB-1994-015-5},
}

@book {Pierce_1982,
    AUTHOR = {Pierce, Richard S.},
     TITLE = {Associative algebras},
    SERIES = {Studies in the History of Modern Science},
    VOLUME = {9},
 PUBLISHER = {Springer-Verlag, New York-Berlin},
      YEAR = {1982},
     PAGES = {xii+436},
      ISBN = {0-387-90693-2},
   MRCLASS = {16-01 (12-01)},

}

@MISC {rschwieb_2016,
    TITLE = {Jacobson radical of extension ring},
    AUTHOR = {rschwieb (https://math.stackexchange.com/users/29335/rschwieb)},
    HOWPUBLISHED = {Mathematics Stack Exchange},
    NOTE = {URL:https://math.stackexchange.com/q/1879263 (version: 2016-08-02)},
    EPRINT = {https://math.stackexchange.com/q/1879263},
    URL = {https://math.stackexchange.com/q/1879263}
}

@article {Swartz_Werner_2021,
    AUTHOR = {Swartz, Eric and Werner, Nicholas J.},
     TITLE = {Covering numbers of commutative rings},
   JOURNAL = {J. Pure Appl. Algebra},
  FJOURNAL = {Journal of Pure and Applied Algebra},
    VOLUME = {225},
      YEAR = {2021},
    NUMBER = {8},
     PAGES = {Paper No. 106622, 17 p.},
      ISSN = {0022-4049,1873-1376},
   MRCLASS = {16P10 (05E16)},

       DOI = {10.1016/j.jpaa.2020.106622},
       URL = {https://doi.org/10.1016/j.jpaa.2020.106622},
}

@article {Tomkinson_1997,
    AUTHOR = {Tomkinson, M. J.},
     TITLE = {Groups as the union of proper subgroups},
   JOURNAL = {Math. Scand.},
  FJOURNAL = {Mathematica Scandinavica},
    VOLUME = {81},
      YEAR = {1997},
    NUMBER = {2},
     PAGES = {191--198},
      ISSN = {0025-5521,1903-1807},
   MRCLASS = {20D60 (20D10)},

       DOI = {10.7146/math.scand.a-12873},
       URL = {https://doi.org/10.7146/math.scand.a-12873},
}

@article {Werner_2015,
    AUTHOR = {Werner, Nicholas J.},
     TITLE = {Covering numbers of finite rings},
   JOURNAL = {Amer. Math. Monthly},
  FJOURNAL = {American Mathematical Monthly},
    VOLUME = {122},
      YEAR = {2015},
    NUMBER = {6},
     PAGES = {552--566},
      ISSN = {0002-9890},
   MRCLASS = {16P10 (05E15)},

       DOI = {10.4169/amer.math.monthly.122.6.552},
       URL = {https://doi.org/10.4169/amer.math.monthly.122.6.552},
}

@article{Scorza_1926,
 author = {Scorza, G.},
 title = {I gruppi che possono pensarsi come somme di tre loro sottogruppi},
 fjournal = {Bollettino della Unione Matematica Italiana},
 journal = {Boll. Unione Mat. Ital.},
 issn = {0041-7084},
 volume = {5},
 pages = {216--218},
 year = {1926},
 zbMATH = {2585723},
 JFM = {52.0113.03}
}

@article{Clark_2012,
 author = {Clark, Pete L.},
 title = {Covering numbers in linear algebra},
 fjournal = {American Mathematical Monthly},
 journal = {Amer. Math. Monthly},
 issn = {0002-9890},
 volume = {119},
 number = {1},
 pages = {65--67},
 year = {2012},

 doi = {10.4169/amer.math.monthly.119.01.065},
 zbMATH = {6071091},
 Zbl = {1245.15004}
}

@article{Ghosh_2022,
 author = {Ghosh, Soham},
 title = {The {Zariski} covering number for vector spaces and modules},
 fjournal = {Communications in Algebra},
 journal = {Comm. Algebra},
 issn = {0092-7872},
 volume = {50},
 number = {5},
 pages = {1994--2017},
 year = {2022},

 doi = {10.1080/00927872.2021.1995741},
 zbMATH = {7502718},
 Zbl = {1483.13034}
}

@article{Gagola_Kappe_2016,
 author = {Gagola III, Stephen M. and Kappe, Luise-Charlotte},
 title = {On the covering number of loops},
 fjournal = {Expositiones Mathematicae},
 journal = {Expo. Math.},
 issn = {0723-0869},
 volume = {34},
 number = {4},
 pages = {436--447},
 year = {2016},

 doi = {10.1016/j.exmath.2015.10.004},
 zbMATH = {6659604},
 Zbl = {1375.20064}
}

@article{Donoven_Kappe_2023,
 author = {Donoven, Casey R. and Kappe, Luise-Charlotte},
 title = {Finite coverings of semigroups and related structures},
 fjournal = {International Journal of Group Theory},
 journal = {Int. J. Group Theory},
 issn = {2251-7650},
 volume = {12},
 number = {3},
 pages = {205--222},
 year = {2023},

 doi = {10.22108/ijgt.2022.131538.1759},
 zbMATH = {7648625},
 Zbl = {1521.20127}
}

@incollection{Kappe_2014,
 author = {Kappe, Luise-Charlotte},
 title = {Finite coverings: a journey through groups, loops, rings and semigroups},
 booktitle = {in: Group theory, combinatorics, and computing, in: Contemp. {M}ath., vol. 611, {A}mer. {M}ath. {S}oc., {P}rovidence, {RI}},
 isbn = {978-0-8218-9435-4; 978-1-4704-1524-2},
 pages = {79--88},
 year = {2014},
volume={611},
 publisher = {Providence, RI: American Mathematical Society (AMS)},
 zbMATH = {6308217},
 Zbl = {1305.20054}
}

@article{Lucchini_Maroti_2012,
 author = {Lucchini, Andrea and Mar{\'o}ti, Attila},
 title = {Rings as the unions of proper subrings},
 fjournal = {Algebras and Representation Theory},
 journal = {Algebr. Represent. Theory},
 issn = {1386-923X},
 volume = {15},
 number = {6},
 pages = {1035--1047},
 year = {2012},

 doi = {10.1007/s10468-011-9277-3},
 zbMATH = {6115376},
 Zbl = {1267.16020}
}

@article{Cai_Werner_2019,
 author = {Cai, Merrick and Werner, Nicholas J.},
 title = {Covering numbers of upper triangular matrix rings over finite fields},
 fjournal = {Involve},
 journal = {Involve},
 issn = {1944-4176},
 volume = {12},
 number = {6},
 pages = {1005--1013},
 year = {2019},

 doi = {10.2140/involve.2019.12.1005},
 zbMATH = {7116066},
 Zbl = {1422.16016}
}

@article{Crestani_2012,
 author = {Crestani, Eleonora},
 title = {Sets of elements that pairwise generate a matrix ring},
 fjournal = {Communications in Algebra},
 journal = {Comm. Algebra},
 issn = {0092-7872},
 volume = {40},
 number = {4},
 pages = {1570--1575},
 year = {2012},

 doi = {10.1080/00927872.2010.544275},
 zbMATH = {6040383},
 Zbl = {1248.16025}
}

@article{Peruginelli_Werner_2018,
 author = {Peruginelli, G. and Werner, N. J.},
 title = {Maximal subrings and covering numbers of finite semisimple rings},
 fjournal = {Communications in Algebra},
 journal = {Comm. Algebra},
 issn = {0092-7872},
 volume = {46},
 number = {11},
 pages = {4724--4738},
 year = {2018},

 doi = {10.1080/00927872.2018.1455099},
 url = {hdl.handle.net/11577/3269898},
 zbMATH = {6959291},
 Zbl = {1398.16018}
}

@article{Brodie_Chamberlain_Kappe_1988,
 author = {Brodie, M. A. and Chamberlain, R. F. and Kappe, L.-C.},
 title = {Finite coverings by normal subgroups},
 fjournal = {Proceedings of the American Mathematical Society},
 journal = {Proc. Amer. Math. Soc.},
 issn = {0002-9939},
 volume = {104},
 number = {3},
 pages = {669--674},
 year = {1988},

 doi = {10.2307/2046770},
 zbMATH = {4130653},
 Zbl = {0691.20019}
}

@article{Neumann_1954,
 author = {Neumann, B. H.},
 title = {Groups covered by permutable subsets},
 fjournal = {Journal of the London Mathematical Society},
 journal = {J. Lond. Math. Soc.},
 issn = {0024-6107},
 volume = {29},
 pages = {236--248},
 year = {1954},

 doi = {10.1112/jlms/s1-29.2.236},
 zbMATH = {3086825},
 Zbl = {0055.01604}
}

@article{Lewin_1967,
 author = {Lewin, Jacques},
 title = {Subrings of finite index in finitely generated rings},
 fjournal = {Journal of Algebra},
 journal = {J. Algebra},
 issn = {0021-8693},
 volume = {5},
 pages = {84--88},
 year = {1967},

 doi = {10.1016/0021-8693(67)90027-0},
 zbMATH = {3230388},
 Zbl = {0143.05303}
}

@article{Swartz_Werner_2024a,
 author = {Swartz, Eric and Werner, Nicholas J.},
 title = {A new infinite family of $\sigma$-elementary rings},
 fjournal = {Communications in Algebra},
 journal = {Comm. Algebra},
 issn = {0092-7872},
 volume = {52},
 number = {1},
 pages = {172--188},
 year = {2024},

 doi = {10.1080/00927872.2023.2237575},
 zbMATH = {7790881}
}

@article{Swartz_Werner_2024b,
    AUTHOR = {Swartz, Eric and Werner, Nicholas J.},
     TITLE = {The covering numbers of rings},
   JOURNAL = {J. Algebra},
  FJOURNAL = {Journal of Algebra},
    VOLUME = {639},
      YEAR = {2024},
     PAGES = {249--280},
      ISSN = {0021-8693,1090-266X},
   MRCLASS = {16P10 (05E16)},

       DOI = {10.1016/j.jalgebra.2023.10.016},
       URL = {https://doi.org/10.1016/j.jalgebra.2023.10.016},
}

@article{ONeill_1980,
 author = {O'Neill, John D.},
 title = {Rings covered by proper ideals},
 fjournal = {Communications in Algebra},
 journal = {Comm. Algebra},
 issn = {0092-7872},
 volume = {8},
 number = {6},
 pages = {505--515},
 year = {1980},
doi = {10.1080/00927878008822472}
}

@article{Bhargava_2002,
 author = {Bhargava, Mira},
 title = {When is a group the union of proper normal subgroups?},
 fjournal = {American Mathematical Monthly},
 journal = {Amer. Math. Monthly},
 issn = {0002-9890},
 volume = {109},
 number = {5},
 pages = {471--473},
 year = {2002},

 doi = {10.2307/2695649},
 zbMATH = {1956324},
 Zbl = {1052.20022}
}

\end{document}